\documentclass[12pt]{amsart}
\usepackage{epsfig,color}

\makeatother

\theoremstyle{definition}

\def\fnum{equation} 
\newtheorem{Thm}[\fnum]{Theorem}
\newtheorem{Cor}[\fnum]{Corollary}

\newtheorem{Lem}[\fnum]{Lemma}

\newtheorem{Exa}[\fnum]{Example}

\newtheorem{Pro}[\fnum]{Proposition}

\numberwithin{equation}{section}

\newcommand{\Vol}{{\text{Vol}}}

\newcommand{\nn}{{\bf{n}}}
\newcommand{\Ric}{{\text{Ric}}}

\newcommand{\Tr}{{\text{Tr}}}

\newcommand{\II}{{\text{II}}}

\newcommand{\Hess}{{\text {Hess}}}

\def\RR{{\bold R}}

\def\SS{{\bold S}}

\newcommand{\dv}{{\text {div}}}

\newcommand{\eqr}[1]{(\ref{#1})}

\title[Harmonic monotonicity]{Ricci curvature and monotonicity for harmonic functions}

\author{Tobias Holck Colding}%
\address{MIT, Dept. of Math.\\
77 Massachusetts Avenue, Cambridge, MA 02139-4307.}

\author{William P. Minicozzi II}%
\address{Johns Hopkins University\\
Dept. of Math.\\
3400 N. Charles St.\\
Baltimore, MD 21218.}

\thanks{The authors
were partially supported by NSF Grants DMS  11040934, DMS
1206827,  and NSF FRG grants DMS 
 0854774 and DMS 0853501}

\email{colding@math.mit.edu and minicozz@math.mit.edu}

\begin{document}

\maketitle

\begin{abstract}
  In this paper we generalize the monotonicity formulas of  \cite{C}  for manifolds with nonnegative Ricci curvature.   Monotone quantities play a key role in   analysis and geometry; see, e.g., 
  \cite{A}, \cite{CM1} and \cite{GL}
  for applications of monotonicity to uniqueness.
  Among the applications here is that level sets of Green's function on open manifolds with nonnegative Ricci curvature are asymptotically umbilic. 
  \end{abstract}

\section{Introduction}

The paper \cite{C} proved three new monotonicity formulas for harmonic functions on manifolds with nonnegative Ricci curvature.  We will see that each of these formulas sits in a two-parameter family of monotonicity formulas, with one parameter corresponding to the equation and the second to the power of the gradient in the formula.  Furthermore, various limiting cases of these formulas carry important geometric information.  This paper deals with only one of the parameters, the power of the gradient; the other will be done in \cite{CM2}.

    \subsection{Monotonicity formulas}  
      
      Our monotonicity formulas will involve a one-parameter family $A_{\beta}$ of scale-invariant 
      quantities on an open manifold $M^n$.  Namely, for   a proper positive function $u:M\setminus \{x\}\to \RR$, define 
\begin{align}
A_{\beta}(r)=r^{1-n} \int_{u=r} |\nabla u|^{1+\beta }\, . 
\end{align}
To simplify notation, given a dimension $n > 2$ and an exponent $\beta \geq \frac{n-2}{n-1} $, then we define
\begin{align}	\label{e:tildeb}
	\tilde{\beta} = \tilde{\beta}(n,\beta) \equiv 1 +(\beta -1) (n-1)  \geq 0 \, .
\end{align}

   The next theorem shows  that $A_{\beta}$ is monotone when $M$ has nonnegative Ricci curvature.
  The derivative is an integral depending on the trace-free second fundamental form $\Pi_0$ 
  of the level sets and the trace-free Hessian $B$ of the function $u^2$ where $u^{2-n}$ is a
   Green's function.\footnote{Our Green's functions will be normalized so that on Euclidean space of dimension $n\geq 3$ the Green's function of the Laplacian is $r^{2-n}$.}

  \begin{Thm}	\label{t:A2mono}
Let $M$ be nonparabolic{\footnote{A complete manifold is   nonparabolic if it admits a positive Green's function $G$. By a result of Varopoulos, \cite{V},   an open manifold with nonnegative Ricci curvature is nonparabolic if and only if 
$
\int_1^{\infty} \frac{r}{\Vol (B_r(x))}\,dr<\infty$. Combining the result of Varopoulos with work of  Li-Yau, \cite{LY}, gives that   $G=G(x,\cdot)\to 0$ at infinity.}}
 with nonnegative Ricci curvature and Green's function $G$.  If
  $ u=G^{\frac{1}{2-n}}$, then
\begin{align}
  A_{\beta} '(r) 
&= -\beta \, r^{n-3} \, \int_{ r \leq u  } u^{4-2n} \,    |\nabla u|^{\beta}
	 \left(         \left| \II_0 \right|^2  
	 +   \Ric (\nn , \nn) 
	  \right) \notag  \\
	  &\quad - \frac{\beta \, r^{n-3}}{4(n-1)} \, \int_{ r \leq u  } u^{2-2n}   \,  |\nabla u|^{\beta -2}   \left( \tilde{\beta} \,
	\left| B(\nn) \right|^2 + (n-2) \, \left| B(\nn)^T \right|^2 
	\right)  \, .
 \end{align}

\end{Thm}

We used $\nn = \frac{\nabla u}{|\nabla u|}$ to denote the unit normal to the level sets of $u$.

\vskip2mm
The other monotonicity formulas show that
various combinations of different $A_{\beta}$'s are monotone.  These will be stated in Section \ref{s:mono}.
 
\subsection{Asymptotically umbilic}

     We say that the level sets of a proper function $u$ are {\emph{asymptotically umbilic}} if 
\begin{align}
r\, \int_r^{2r} \, \frac{1}{\Vol (u=s)}\int_{u=s}\left|  \II_0 \right|^2 \, ds \to 0 \text{ as } r\to \infty\, .
\end{align}
 
A consequence of our monotonicity formulas is that the level sets of $u$ are asymptotically umbilic if $M$ has Euclidean volume growth; cf. \cite{ChCT}.

\begin{Cor}	\label{c:umb}
If $M$ has nonnegative Ricci curvature and Euclidean volume growth, i.e., $\lim_{r\to \infty} \, r^{-n} \, \Vol (B_r) > 0$, then the level sets of $ u=G^{\frac{1}{2-n}}$ are
 asymptotically umbilic.
\end{Cor}

\section{The trace-free Hessian}

In this section, we will compute the Laplacian on combinations of $u$ and $|\nabla u|$ that come up later in the monotonicity formulas.  In order to squeeze as much as possible out of these formulas, we will express them in terms of the trace-free Hessian $B$ and trace-free second fundamental form $\Pi_0$ which vanish in the model case.

\vskip2mm
We will several times later use the following elementary lemma:

\begin{Lem}	\label{l:}
Let  $u:M\to \RR$ be a smooth positive function.  
The following are equivalent:
\begin{itemize}
\item $\Delta\, u^2=2n\,|\nabla u|^2$.
\item $\Delta\, u=(n-1)\,\frac{|\nabla u|^2}{u}$.
\item $\Delta\, u^{2-n}=0$.
\end{itemize}
\end{Lem}

\begin{proof}
For any positive function $v$ and any real number $\alpha$ we have the following
\begin{equation}
	\Delta v^{\alpha}=\alpha\, v^{\alpha-1}\,\left( (\alpha-1)\, \frac{|\nabla v|^2}{v}+\Delta\, v\right)\, .
\end{equation}
The lemma easily follows.
\end{proof}

\subsection{The trace free Hessian of $u^2$}
Throughout this section,  the function $u$ satisfies 
\begin{equation}
	\Delta u^2 = 2n \, |\nabla u|^2 \, . 
\end{equation}
 The basic example  is the function $|x|$ on Euclidean space $\RR^n$.  In that case, the full Hessian of $|x|$ satisfies a nice equation, not just the Laplacian of $|x|$.  There are a number of ways of writing this, perhaps the simplest being that the trace-free Hessian of $|x|^2$ vanishes.  
 
 With this in mind, 
we 
define the tensor $B$ to be the trace-free    Hessian of $u^2$ 
\begin{align}
	B =   \Hess_{u^2}   - 2\, |\nabla u|^2 \, g  \, ,
\end{align}
where $g$ is the Riemannian metric.

We will use that $ \Hess_{u^2} = 2 \, u \, \Hess_u + 2 \, \nabla u \otimes \nabla u$, so that
\begin{align}	\label{e:uhessu}
	2 \, u \, \Hess_u =  \Hess_{u^2} - 2 \, \nabla u \otimes \nabla u = B + 2\, \left(   |\nabla u|^2 \, g  
	-   \nabla u \otimes \nabla u  \right) \, .
\end{align}

The next lemma computes the gradient of $|\nabla u|^2$ in terms of $B$.

\begin{Lem}	\label{l:k1}
We have   $u \, \nabla |\nabla u|^2 = B(\nabla u)$, where $B(\nabla u)$ is   given by
$\langle B(\nabla u) , v \rangle \equiv B(\nabla u , v)$. 
\end{Lem}

\begin{proof}
Since $\nabla |\nabla u|^2 = 2\, \Hess_u (\nabla u , \cdot )$, equation \eqr{e:uhessu} gives
\begin{align}
	u \, \nabla |\nabla u|^2 = 2\, u \, \Hess_u (\nabla u , \cdot ) = 
	 B (\nabla u , \cdot) + 2\, \left(   |\nabla u|^2 \, \nabla u  
	-  |\nabla u|^2 \,  \nabla u  \right) = B(\nabla u , \cdot) \, .
\end{align}
\end{proof}

\begin{Cor}	\label{c:k1}
We have    $2\, u \, \nabla |\nabla u| =   B(\nn)$ where $\nn = \frac{\nabla u}{|\nabla u|}$
 and $4 \, u^2 \, \left| \nabla |\nabla u| \right|^2  = \left| B(\nn)  \right|^2$.
\end{Cor}

\begin{proof}
Since $u \, \nabla |\nabla u|^2 = 2 \, u \, |\nabla u| \, \nabla |\nabla u|$, this
  follows from Lemma \ref{l:k1}.
\end{proof}

The next lemma computes the  divergence of $B$.

\begin{Lem}	\label{l:k2}
The divergence of $B$ is 
\begin{equation}
	\delta B = \Ric (\nabla u^2 , \cdot) + (2n-2) \, \nabla |\nabla u|^2 = 
	 \Ric (\nabla u^2 , \cdot) + (2n-2) \,  u^{-1} \, B( \nabla u) \, .
\end{equation}
\end{Lem}

\begin{proof}
Fix a point $p \in M$ and let $e_i$ be an orthonormal frame at $p$ with $\nabla_{e_i} e_j (p) = 0$.  
Given a function $w$, the symmetry of the Hessian and the definition of the curvature give
\begin{align}	\label{e:ricciw}
	w_{ijj} = w_{jij} = w_{jji} + \Ric_{ij} w_j \, .
\end{align}
Using the definition of $B$, the fact that $g$ is parallel, and \eqr{e:ricciw} with $w=u^2$ gives
\begin{align}
	\left( \delta B \right)_i \equiv B_{ij,j} = (u^2)_{ijj} - 2\left( |\nabla u|^2 \right)_i = \Ric_{ij} \, (u^2)_j + \left( \Delta u^2 \right)_i 
	- 2\left( |\nabla u|^2 \right)_i \, .
\end{align}
Thus,   $\delta B = \Ric (\nabla u^2 , \cdot ) +  \nabla ( \Delta u^2 - 2 \, |\nabla u|^2)$.
The lemma follows since $\Delta u^2 = 2n \, |\nabla u|^2$.
\end{proof}

Using this, we can compute the Laplacian of $|\nabla u|^2$.

\begin{Lem}		\label{l:deltanu}
We have 
\begin{align}
	u^2 \, \Delta |\nabla u|^2  &= \frac{1}{2} \, \left| B \right|^2 + (2n-4) \,    B (\nabla u , \nabla u)
	+ \frac{1}{2} \, \Ric (\nabla u^2 , \nabla u^2)  \notag \\
	&= \frac{1}{2} \, \left| B \right|^2 + (n-2) \,   \langle \nabla |\nabla u|^2 , \nabla u^2 \rangle
	+ \frac{1}{2} \, \Ric (\nabla u^2 , \nabla u^2) 
	\, .
\end{align}
\end{Lem}

\begin{proof}
Using the definition of the Laplacian, then Lemma \ref{l:k1}, and then Lemma \ref{l:k2} gives
\begin{align}
	u^2 \, \Delta |\nabla u|^2 &= u^2 \, \dv  \, \nabla |\nabla u|^2 = u^2\,  \dv \, \left( u^{-1} \, B(\nabla u) \right)   = u\, \langle \delta B , \nabla u \rangle + 
	   \langle B , u\, \Hess_u \rangle -   B (\nabla u , \nabla u) \notag \\
	 &= 
	 u \, \Ric (\nabla u^2 , \nabla u) + 
	 (2n-2) \, B(\nabla u  , \nabla u) \\
	 &\qquad +
	\langle B , \left\{ \frac{1}{2} \, B  +  \left(   |\nabla u|^2 \, g  
	-   \nabla u \otimes \nabla u  \right)  \right\} \rangle -   B (\nabla u , \nabla u) \,  . \notag
\end{align}
Using that $\langle B , g \rangle = 0$ (since $B$ is trace-free)  and  
$\langle B , \nabla u \otimes \nabla u \rangle = B(\nabla u , \nabla u)$ gives
\begin{align}
	u^2 \, \Delta |\nabla u|^2  = \frac{1}{2} \, \Ric (\nabla u^2 , \nabla u^2) + 
	 (2n-4) \, B(\nabla u , \nabla u) + 
	\frac{1}{2} \, \left| B \right|^2  
	\, .  \notag
\end{align}
 This gives the first equality.  The second   uses that  $u \, \nabla |\nabla u|^2 = B(\nabla u)$ by Lemma \ref{l:k1}.
\end{proof}

More generally, we  compute the Laplacian of powers of $|\nabla u|$.
 
 \begin{Pro}		\label{p:deltanu}
If   $\alpha \ne 0$, then
\begin{align}
	\frac{2}{\alpha} \,|\nabla u|^{2-\alpha} \, 
	 \Delta |\nabla u|^{\alpha}  = \frac{ \left| B \right|^2 
	 +  (\alpha - 2) \,  \left| B(\nn) \right|^2  + \Ric (\nabla u^2 , \nabla u^2) }{2u^2} \, 
	 + \frac{n-2}{u^2} \, \langle \nabla |\nabla u|^2 , \nabla u^2 \rangle 
	\, .
\end{align}
\end{Pro}

\begin{proof}
We have
\begin{align}
	2\, \nabla \, |\nabla u|^{\alpha} = 2\, \alpha \, |\nabla u|^{\alpha - 1} \, \nabla |\nabla u|
	=   \alpha \, |\nabla u|^{\alpha - 2} \, \nabla |\nabla u|^2   \, .
\end{align}
Taking the divergence of this and using Lemma \ref{l:deltanu} gives
\begin{align}
	\frac{2}{\alpha} \, |\nabla u|^{3-\alpha  }  \, \Delta |\nabla u|^{\alpha}  &= 
	|\nabla u|^{3-\alpha  }  \,
	\dv \, \left( |\nabla u|^{\alpha - 2} \, \nabla |\nabla u|^2  \right)  \notag \\
	&= 
	|\nabla u| \,  \Delta |\nabla u|^2  + (\alpha - 2) \, \langle \nabla |\nabla u|^2 , \nabla |\nabla u| \rangle  \\
	&= \frac{|\nabla u|}{2u^2} \, \Ric (\nabla u^2 , \nabla u^2) +
	 \frac{|\nabla u|}{2u^2} \, |B|^2 + (n-2) \,  \frac{|\nabla u|}{u^2} \, \langle \nabla |\nabla u|^2 , \nabla u^2 \rangle \notag \\
	 &\qquad
	+ 2\, (\alpha - 2) \,|\nabla u| \, \left| \nabla |\nabla u| \right|^2
	\, . \notag 
\end{align}
Since $4 \, u^2 \, \left| \nabla |\nabla u| \right|^2  = \left| B(\nn)  \right|^2$ by Corollary \ref{c:k1}, 
simplifying this gives the proposition.
\end{proof}

When $\alpha = 1$, we get the following corollary:

 \begin{Cor}		\label{c:deltanu}
We have
\begin{align}
	2 \,|\nabla u| \, 
	 \Delta |\nabla u|   = \frac{1}{2u^2} \,  \left( \left| B \right|^2 
	  -   \left| B(\nn) \right|^2
	  + \Ric (\nabla u^2 , \nabla u^2)
	   \right)
	 + \frac{n-2}{u^2} \, \langle \nabla |\nabla u|^2 , \nabla u^2 \rangle 
	\, .
\end{align}
\end{Cor}

 \subsection{The trace-free second fundamental form}
 
The second fundamental form $\II$ of the level sets of $u$ is given by
\begin{align}
	\II (e_i , e_j) \equiv \langle \nabla_{e_i} \nn , e_j \rangle \, , 
\end{align}
where $e_i$ is a tangent frame and $\nn = \frac{ \nabla u}{|\nabla u|}$ is the unit normal. 

\begin{Lem}	\label{l:tracefree}
 The trace-free second fundamental form $\II_0$ is given by
\begin{align}
	2 \, u \, |\nabla u| \,  \II_0  = B_0  + \frac{B(\nn , \nn)}{n-1}     \, g_0 \, ,
\end{align}
where    $B_0$ is the restriction of $B$ to  tangent vectors and $g_0$ is the metric on the level set.
\end{Lem}

\begin{proof}
 Using that $\nabla u$ is normal, we can rewrite $\II$ as
\begin{align}
	2 \, u \, |\nabla u| \,  \II (e_i , e_j) = \langle \nabla_{e_i}   \nabla u^2   , e_j \rangle
	= \Hess_{u^2} (e_i , e_j)  \, .
\end{align}
The mean curvature $H$ is the trace of $\II$ over the $e_i$'s.  We have
\begin{align}
	2 \, u \, |\nabla u| \, H &= \Delta u^2 - \Hess_{u^2} (\nn  , \nn) = 2n \, |\nabla u|^2 - \Hess_{u^2} (\nn  , \nn) \notag \\
	&= 2(n-1) \, |\nabla u|^2 + \left( 2\, |\nabla u|^2 -  \Hess_{u^2} (\nn  , \nn)  \right) \\
	&= 2(n-1) \, |\nabla u|^2  - B(\nn , \nn) \, . \notag
\end{align}
Thus,  the trace-free second fundamental form $\II_0$ is given by
\begin{align}
	2 \, u \, |\nabla u| \,  \II_0 &= 2 \, u \, |\nabla u| \,  \left(  \II - \frac{H}{n-1} \, g_0\right)  = 
	\Hess_{u^2}   -   2\, |\nabla u|^2 \, g_0 + \frac{B(\nn , \nn)}{n-1}     \, g_0 \notag \\
	&= B_0  + \frac{B(\nn , \nn)}{n-1}     \, g_0 \, ,
\end{align}
where  $\Hess_{u^2}$ is restricted to   tangent vectors.
 \end{proof}

 \begin{Lem}	\label{l:B0}
We have 
 \begin{align}
 	4 \, u^2 \, |\nabla u|^2 \, \left| \II_0 \right|^2 &=   |B_0|^2 - \frac{(B(\nn , \nn))^2}{n-1}
	 =  \left| B \right|^2 
	  -   \frac{n}{n-1} \, \left| B(\nn)  \right|^2  -   \frac{n-2}{n-1} \, \left| B(\nn)^T \right|^2 \, .
 \end{align}
 \end{Lem}
 
 \begin{proof}
Since $B$ is trace-free, we get that
 \begin{equation}
 	\langle B_0 , g_0 \rangle =   \Tr (B) - B(\nn , \nn) = - B(\nn , \nn) \, , 
 \end{equation}
 Using this in   Lemma \ref{l:tracefree} 
gives
 \begin{align}
 	4 \, u^2 \, |\nabla u|^2 \, \left| \II_0 \right|^2 &=   
	 \left| B_0 \right|^2  + \frac{\left(  B(\nn , \nn)  \right)^2}{ (n-1)^2} \, |g_0|^2  + 2\, \frac{B(\nn , \nn)}{n-1} \langle B_0 , g_0 \rangle \notag \\
	 &=   \left| B_0 \right|^2  + \frac{\left(  B(\nn , \nn)  \right)^2}{ n-1} - 2\, \frac{(B(\nn , \nn))^2}{n-1}
	 \, .
 \end{align}
 This gives the first equality.  To get the second equality,  use the symmetry of $B$ to get
 \begin{align}
 	\left| B \right|^2 &= \sum_{i,j \leq (n-1)} \, \, (B(e_i , e_j))^2 + 2 \, \sum_{i \leq (n-1)} (B(\nn , e_i))^2 + (B(\nn , \nn))^2 \notag \\
	& = \left| B_0 \right|^2 + 2\, \left| B(\nn)^T \right|^2 + (B(\nn , \nn ))^2  
\end{align}
and note that $\left| B(\nn) \right|^2 = \left|  B(\nn)^T \right|^2 + (B(\nn , \nn))^2$.
 \end{proof}

   \subsection{Divergence formulas}
   
   We will compute the divergence of various quantities involving $u$ and $|\nabla u|$.  
   We will need the following differential inequality for $|\nabla u|$.
 
 \begin{Pro}	\label{p:ptog}  
 If   $\II_0$ is the trace-free second fundamental form of the level set,   then 
 \begin{align}
	 \Delta |\nabla u|   &=   |\nabla u|        \left| \II_0 \right|^2  + \frac{ \Ric (\nabla u , \nabla u)}{|\nabla u|}
	 + \frac{n-2}{u^2} \, \langle \nabla |\nabla u| , \nabla u^2 \rangle \notag \\
	 &\qquad +   \frac{ \left(    \left| B(\nn) \right|^2 + (n-2) \, \left| B(\nn)^T \right|^2  \right)}{4(n-1) \, |\nabla u| \, u^2} 
	\, .
\end{align}
 \end{Pro}
 
 \begin{proof}
Corollary \ref{c:deltanu}
gives that
\begin{align}
	|\nabla u| \, 
	 \Delta |\nabla u|   &= \frac{1}{4u^2} \,  \left( \left| B \right|^2 
	  -   \left| B(\nn) \right|^2 + \Ric (\nabla u^2 , \nabla u^2) \right)  
	  	 + \frac{n-2}{u^2} \,|\nabla u|  \langle \nabla |\nabla u| , \nabla u^2 \rangle 
	\, .
\end{align}
The next ingredient is Lemma \ref{l:B0} which
gives that
 \begin{align}
		4 \, u^2 \, |\nabla u|^2 \, \left| \II_0 \right|^2 &= \left(   \left| B \right|^2 -  \left| B(\nn)  \right|^2 \right) - \frac{ \left| B(\nn)  \right|^2}{n-1} - \frac{n-2}{n-1} \, \left| B(\nn)^T \right|^2 \, ,
 \end{align}
 so we see that
 \begin{align}
 	\frac{ \left( \left| B \right|^2 
	  -   \left| B(\nn) \right|^2 \right)}{4\, |\nabla u| \, u^2}  = |\nabla u| \, \left| \II_0 \right|^2
	  + \frac{ \left(    \left| B(\nn) \right|^2 + (n-2) \, \left| B(\nn)^T \right|^2  \right)}{4(n-1) \, |\nabla u| \, u^2} 
 \end{align}
 \end{proof}

 \begin{Lem}	\label{l:divforms}
Given $p , \alpha \in \RR$, we have
 \begin{align}
 	&\dv \, \left( u^{2p} \, |\nabla u|^{\alpha} \, \nabla u^2 \right)  =
	(2n+4p) \, u^{2p} \, |\nabla u|^{2+ \alpha}  + \alpha \, u^{2p} \, |\nabla u|^{\alpha - 1} \, 
	\langle \nabla |\nabla u| , \nabla u^2 \rangle   \, , \\
	&\dv \, \left( u^{2p} \, |\nabla u|^{\alpha} \nabla |\nabla u| \right) =   
	u^{2p-2} \, (p+n -2)\, |\nabla u|^{\alpha} \, 
	\langle \nabla u^2 , \nabla |\nabla u| \rangle \\
	& + u^{2p} \,  |\nabla u|^{1+\alpha}    \left(    \left| \II_0 \right|^2  + \Ric (\nn , \nn) 
	+   \frac{   (1+\alpha (n-1))  \left| B(\nn) \right|^2 + (n-2) \, \left| B(\nn)^T \right|^2  }{4(n-1) \, |\nabla u|^2 \, u^2} \right)  \notag \, .
 \end{align}
 \end{Lem}
 
 \begin{proof}
 For the first claim, use $\Delta u^2 = 2n \, |\nabla u|^2$ to compute
  \begin{align}
 	\dv \, \left( u^{2p} \, |\nabla u|^{\alpha} \, \nabla u^2 \right) & = p \, u^{2p-2} \, |\nabla u|^{\alpha} \, 
	\left| \nabla u^2 \right|^2 + u^{2p} \, |\nabla u|^{\alpha} \Delta u^2 + \alpha \, u^{2p} \, |\nabla u|^{\alpha - 1} \, 
	\langle \nabla |\nabla u| , \nabla u^2 \rangle    \notag \\
 	&=(2n+4p) \, u^{2p} \, |\nabla u|^{2+ \alpha}  + \alpha \, u^{2p} \, |\nabla u|^{\alpha - 1} \, 
	\langle \nabla |\nabla u| , \nabla u^2 \rangle    \, .
 \end{align}
 To get the second claim, first
use Proposition \ref{p:ptog} to compute
 \begin{align}
 	\dv \, \left( u^{2p} \, \nabla |\nabla u| \right) &=  u^{2p} \, \Delta |\nabla u| +
	p \, u^{2p-2} \, \langle \nabla u^2 , \nabla |\nabla u| \rangle 
	\notag \\
	&=
	  u^{2p} \left(   |\nabla u|        \left| \II_0 \right|^2  + |\nabla u| \,  \Ric (\nn, \nn) 
	 + \frac{n-2}{u^2} \, \langle \nabla |\nabla u| , \nabla u^2 \rangle \right. \notag \\
	 &\qquad + \left. 
	   \frac{ \left(    \left| B(\nn) \right|^2 + (n-2) \, \left| B(\nn)^T \right|^2  \right)}{4(n-1) \, |\nabla u| \, u^2}  \right) +
	p \, u^{2p-2} \, \langle \nabla u^2 , \nabla |\nabla u| \rangle \\
	&=  u^{2p} \,  |\nabla u|    \left(    \left| \II_0 \right|^2  + \Ric (\nn , \nn)
	+   \frac{ \left(    \left| B(\nn) \right|^2 + (n-2) \, \left| B(\nn)^T \right|^2  \right)}{4(n-1) \, |\nabla u|^2 \, u^2}  \right) 
	\notag \\
	&\qquad
	+  u^{2p-2} \, (p+n -2)\,
	\langle \nabla u^2 , \nabla |\nabla u| \rangle
	\, . \notag 
\end{align}
The second claim follows from this since
 \begin{align}
 	\dv \, \left( u^{2p} \, |\nabla u|^{\alpha} \nabla |\nabla u| \right) &= |\nabla u|^{\alpha} \, 
	\dv \, \left( u^{2p} \,  \nabla |\nabla u| \right)
	+ \alpha \, \, |\nabla u|^{\alpha - 1} \,  
	   u^{2p} \left|  \nabla |\nabla u| \right|^2  \notag \\
	    &= |\nabla u|^{\alpha} \, 
	\dv \, \left( u^{2p} \,  \nabla |\nabla u| \right)
	+ \frac{\alpha}{4}  \, \, |\nabla u|^{\alpha - 1}  
	   u^{2p-2}  \,  \left| B(\nn ) \right|^2
		\, ,
\end{align}
where the last equality used  that $4\, u^2 \, \left| \nabla |\nabla u| \right|^2 = \left| B(\nn ) \right|^2 $ by Corollary \ref{c:k1}.
 \end{proof}
 
 The previous divergence formulas allow us next to compute the Laplacian on various combinations of $u$ and $|\nabla u|$.  Recall that $\tilde{\beta} \geq 0$ was defined in \eqr{e:tildeb}.
 
 \begin{Pro}	\label{p:dvs}
 Given $q, \beta \in \RR$, we have
 \begin{align}
 	\Delta \, \left( u^{2q} \, |\nabla u|^{\beta} \right) & =	2\,q\, (2q+n-2) \, u^{2q-2} \, |\nabla u|^{2+\beta}   + \beta \,  
	 u^{2q} \, |\nabla u|^{\beta} \left( \left| \II_0 \right|^2 + \Ric (\nn , \nn)    \right) \notag \\
	&\qquad + \frac{ \beta}{4(n-1)} \,  
	 u^{2q-2} \, |\nabla u|^{\beta-2} 
	 \, \left( \tilde{\beta} \, \left| B(\nn )\right|^2 + (n-2) \, \left| B(\nn)^T \right|^2
	\right)   \\
	&\qquad + \beta \, (2q+n-2) \,    u^{2q-2}  \, |\nabla u|^{\beta - 1} \langle \nabla |\nabla u| , \nabla u^2 \rangle
	  \, . \notag 
 \end{align}
 \end{Pro}
 
 \begin{proof}
 The gradient is given by 
 \begin{align}
 	\nabla \left( u^{2q} \, |\nabla u|^{\beta} \right)  = q\,  u^{2q-2} \, |\nabla u|^{\beta} \nabla u^2  
	+ \beta \, u^{2q} \, |\nabla u|^{\beta - 1} \, \nabla |\nabla u| \, .
\end{align}
Taking the divergence of this and then applying the first claim in Lemma \ref{l:divforms} with
$p=q-1$ and $\alpha = \beta$ and the second claim there with $q=p$ and $\alpha = \beta -1$
gives
  \begin{align}
 	\Delta \, \left( u^{2q} \, |\nabla u|^{\beta} \right) & = 
	q\, \dv \, \left( u^{2q-2} \, |\nabla u|^{\beta} \nabla u^2  \right)
	+ \beta \, \dv \left( u^{2q} \, |\nabla u|^{\beta - 1} \, \nabla |\nabla u|
	\right) \notag \\
	&= 
	q\, \left\{  (2n+4q-4) \, u^{2q-2} \, |\nabla u|^{2+\beta} + \beta \, u^{2q-2} \, |\nabla u|^{\beta -1} \, \langle \nabla |\nabla u| , \nabla u^2 \rangle 
	\right\} \\
	&\qquad + \beta \,  
	 u^{2q} \, |\nabla u|^{\beta} \left( \left| \II_0 \right|^2 + \Ric (\nn , \nn)    \right)
	 + \beta \,   u^{2q-2} \, (q+n-2) \, |\nabla u|^{\beta - 1} \langle \nabla |\nabla u| , \nabla u^2 \rangle \notag \\
	&\qquad + \frac{ \beta \,  
	 u^{2q-2} \, |\nabla u|^{\beta-2} }{4(n-1)} 
	 \, \left(   \tilde{\beta} \, \left| B(\nn )\right|^2 + (n-2) \, \left| B(\nn)^T \right|^2
	\right) 
	  \, . \notag 
 \end{align}
 \end{proof}
 
We separately record the cases   $2q= 2-n$ and $q=0$ next.
  
 \begin{Cor}	\label{c:optimize}
We have
 \begin{align}
 	\Delta \, \left( u^{2-n} \, |\nabla u|^{\beta} \right) &=
	\beta \,  
	 u^{2-n} \, |\nabla u|^{\beta} \left( \left| \II_0 \right|^2 + \Ric (\nn , \nn)    \right) \notag \\
	&\qquad + \frac{ \beta}{4(n-1)} \,  
	 u^{-n} \, |\nabla u|^{\beta-2} 
	 \, \left(   \tilde{\beta} \, \left| B(\nn )\right|^2 + (n-2) \, \left| B(\nn)^T \right|^2
	\right) 
	   \, , \\
	\Delta \, |\nabla u|^{\beta} &=
	\beta \,  
	  |\nabla u|^{\beta} \left( \left| \II_0 \right|^2 + \Ric (\nn , \nn)    \right) 
	  + \beta \, (n-2) \,    u^{-2}  \, |\nabla u|^{\beta - 1} \langle \nabla |\nabla u| , \nabla u^2 \rangle  \notag \\
	&\qquad + \frac{ \beta}{4(n-1)} \,  
	 u^{-2} \, |\nabla u|^{\beta-2} 
	 \, \left(  \tilde{\beta}  \, \left| B(\nn )\right|^2 + (n-2) \, \left| B(\nn)^T \right|^2
	\right)   		\, .  
 \end{align}
 \end{Cor}
 
 \begin{proof}
Take $2q=2-n$ and $q=0$ in Proposition \ref{p:dvs}. 
 \end{proof}
 
 \section{The monotonicity formulas}		\label{s:mono}
 
 In this section, we use the formulas from the previous section to prove three one-parameter families of monotonicity formulas generalizing the three formulas from \cite{C}.  The parameter corresponds to the power of $|\nabla u|$ in the integrand and monotonicity holds as long as the parameter is above a certain critical value.  We will see in \cite{CM2} that these formulas sit in two-parameter families where the functions   then   satisfy the nonlinear $p$-Laplace equation.

 Throughout this section, $u>0$   satisfies $\Delta u^2 = 2n \, |\nabla u|^2$ as in the previous section and, in addition, is proper and is normalized so that 
 \begin{align}
 	\lim_{r\to 0}  \, \frac{u}{r} = 1 \, , 
 \end{align}
 where $r$ is the distance to a fixed point.
 
 Recall that the scale-invariant quantity $A_{\beta}$ is given by
 \begin{align}
 	A_{\beta} (r) &= r^{1-n} \, \int_{  u = r } |\nabla u|^{1+\beta} \, .
\end{align}
We also define a second family of scale-invariant quantities $V_{\beta}$ by
 \begin{align}
	V_{\beta} (r)= r^{2-n} \, \int_0^r \, \int_{ u = s} \frac{ |\nabla u|^{1+\beta}}{u^2} \, ds =r^{2-n} \, \int_{ u \leq r} \frac{ |\nabla u|^{2+\beta}}{u^2}\, ,
\end{align}
where the second equality is the
 co-area formula.  
Differentiating $V_{\beta}$ gives
\begin{align}	\label{l:V1beta} 
	r\, V_{\beta} ' (r) =  (2-n)  \, V_{\beta} (r) +  A_{\beta} (r) \, .
\end{align} 

The following simple lemma shows that both $A_{\beta}$ and $V_{\beta}$ are uniformly bounded:
 
 \begin{Lem}   \label{l:bounded}
 If $M$ is nonparabolic with nonnegative Ricci curvature, $G$ is a Green's function, 
  and $u=G^{ \frac{1}{2-n} }$, then for all $r$
   \begin{align}
 	    A_{\beta}(r) &\leq  \Vol (\partial B_1(0)) \leq 
	    r^{1-n} \, \Vol (u=r) \, ,\\
	    V_{\beta} (r) &\leq  \frac{\Vol (\partial B_1(0))}{n-2}\, .
	     \end{align}
	     In particular, both are uniformly bounded by their Euclidean values.
 \end{Lem}
 
 \begin{proof}
 As in  \cite{C}, we have that
 $
 |\nabla u|\leq 1$ and
 \begin{align}
 r^{1-n} \, \int_{ u = r } |\nabla u|&=\Vol (\partial B_1(0))\, . 
 \end{align}
 Both claims   follow easily from this. 
  \end{proof}

 \subsection{The  first and  second monotonicity formulas}
 
 We will now state and prove the first two monotonicity formulas.  Recall that the constant $\tilde{\beta} \equiv 1 + (\beta -1) (n-1)$ is nonnegative since $n> 2$ and $\beta \geq \frac{n-2}{n-1}$.

\vskip2mm
The next theorem gives the first monotonicity formula.

 \begin{Thm}	\label{t:mono2}
 We have 
  \begin{align}
 	   &\left( A_{\beta}
	     -2\, 
	       (  n-2)  \, V_{\beta} \right)'(r) = 
	      \frac{\beta}{r^{n-1}} \, \int_{0}^r \int_{u = s } |\nabla u|^{\beta-1} \left( \left| \II_0 \right|^2 + \Ric (\nn , \nn)
	   \right) \, ds  \notag \\
	   &\qquad + 
	   \frac{ \beta}{4(n-1)} \,  r^{1-n} \int_{0}^r \int_{u=s}
	 u^{-2} \, |\nabla u|^{\beta-3} 
	 \, \left( \tilde{\beta} \, \left| B(\nn )\right|^2 + (n-2) \, \left| B(\nn)^T \right|^2
	\right) \, ds  \, 
 \end{align}
 \end{Thm}

  \begin{proof}
 Differentiating $A_{\beta}$ as in \cite{CM3}, \cite{CM4} gives 
  \begin{align}	\label{e:abetap}
 	 A_{\beta}'(r) &=     r^{1-n} \, \int_{ u = r } \langle \nabla   |\nabla u|^{\beta} , \nn \rangle \, .
\end{align}
Applying Stokes' theorem, noting that the interior boundary integral goes to zero (using the asymptotics of
the Green's function at the pole; cf. \cite{C}), and  
then using Corollary \ref{c:optimize} gives
  \begin{align}
 	r^{n-1} \, A_{\beta}'(r) &=      
	\int_{ u \leq r } \Delta  |\nabla u|^{\beta} \notag \\
	 &=     \int_{ u \leq r  } \left( \beta \,  
	  |\nabla u|^{\beta} \left( \left| \II_0 \right|^2 + \Ric (\nn , \nn)    \right) +
	  \beta \, (n-2) \,    u^{-2}  \, |\nabla u|^{\beta - 1} \langle \nabla |\nabla u| , \nabla u^2 \rangle  \right) \\
	  &+ \frac{ \beta}{4(n-1)} \,   \int_{u \leq r}  
	 u^{-2} \, |\nabla u|^{\beta-2} 
	 \, \left( \tilde{\beta}\, \left| B(\nn )\right|^2 + (n-2) \, \left| B(\nn)^T \right|^2
	\right) \, . \notag 
 \end{align}
 We apply the coarea formula on every term except the last term on the second line.
 For this,  we use the first claim in 
Lemma \ref{l:divforms} to get
 \begin{align}
 	\dv \, \left(   |\nabla u|^{\beta}  \, \frac{\nabla u^2}{u^2}  \right) =
	 \frac{\beta |\nabla u|^{\beta-1}}{u^2} \, \langle \nabla |\nabla u| , \nabla u^2 \rangle + (2n-4) \, 
	\frac{  \left| \nabla u \right|^{2+\beta}}{u^2}  \, .
 \end{align}
 Since $n>2$, the interior boundary integral goes to zero and
  the divergence theorem gives
 \begin{align}
   \int_{ u \leq r  }   \frac{\beta |\nabla u|^{\beta - 1}}{u^2} \, \langle \nabla |\nabla u| , \nabla u^2 \rangle   &=
    \frac{2}{r} \, \int_{ u = r }  |\nabla u|^{1+\beta}   -
   (2n-4) \,  \int_{ u \leq r }  \frac{  \left| \nabla u \right|^{2+\beta}}{u^2}  \notag \\
   & =  2 \, r^{n-1} \, \left( \frac{ A_{\beta} (r) + (2-n) \,   V_{\beta} (r) }{r} \right) =
   2 \, r^{n-1} \, V_{\beta}'(r) \, ,
 \end{align}
 where the last equality used \eqr{l:V1beta}.
 Multiplying by $(n-2) \, r^{1-n}$ and 
 putting this back into the formula for $A_{\beta}'(r)$ gives
 the theorem.
 \end{proof}
 
  The case $\beta = 2$ in Theorem \ref{t:mono2} is the first monotonicity formula in \cite{C}.  We record the case $\beta = 1$ below
 separately as it seems to be of particular significance.    
 
  \begin{Cor}	\label{c:mono1}
 We have
 \begin{align}
 	 \left( A_1 - 2(n-2)\, V_1 \right)' (r)  &\geq  
	r^{1-n} \, \int_{0}^r \int_{u = s} \left( \left| \II_0 \right|^2 + \Ric (\nn , \nn)
	   \right) \, ds 
	 \, .
 \end{align}
 \end{Cor}

   Our second monotonicity formula follows.

 \begin{Thm}	\label{t:mono3}
 If  $r_1 < r_2$, then 
   \begin{align}
   	&(2-n) \,  A_{\beta} (r_2)  + r_2 \,  A_{ \beta} '(r_2)
	- (2-n) \,  A_{ \beta} (r_1)  - r_1 \,  A_{ \beta} '(r_1) \notag \\
	&\qquad  =   \beta \,   \int_{r_1 \leq u \leq r_2} 
	 u^{2-n} \, |\nabla u|^{\beta} \left( \left| \II_0 \right|^2 + \Ric (\nn , \nn)    \right) \notag \\
	 &\qquad + \frac{ \beta}{4(n-1)} \,  \int_{r_1 \leq u \leq r_2}
	 u^{-n} \, |\nabla u|^{\beta-2} 
	 \, \left(   \tilde{\beta} \, \left| B(\nn )\right|^2 + (n-2) \, \left| B(\nn)^T \right|^2
	\right) \, .
\end{align}
 Moreover, we have
\begin{align}
	  &  \left[r^{2-n} \, \left( A_{ \beta} (r) -  \Vol (\partial B_1(0))\right)  \right]' 
	  = \beta \,  r^{1-n} \,  \int_{u \leq r} 
	 u^{2-n} \, |\nabla u|^{\beta} \left( \left| \II_0 \right|^2 + \Ric (\nn , \nn)    \right) \notag \\
	 &\qquad + \frac{ \beta}{4(n-1)} \,  \int_{r_1 \leq u \leq r_2}
	 u^{-n} \, |\nabla u|^{\beta-2} 
	 \, \left(   \tilde{\beta} \, \left| B(\nn )\right|^2 + (n-2) \, \left| B(\nn)^T \right|^2
	\right) \, ,
  \end{align}
where $B_1(0)$ is the Euclidean ball of radius one.
 \end{Thm}

  \begin{proof} 
  To keep notation simple, within this proof define $f(r)$ and $g(r)$  by
  \begin{align}
  	f(r) &\equiv r^{2-n} \, A_{\beta}(r) = r^{1-n} \int_{u=r} u^{2-n} \, |\nabla u|^{1+ \beta} \, , \\
	g(r) &\equiv r^{1-n} \, f'(r) =
	 r^{n-1} \, \left(r^{2-n} \, A_{ \beta} (r) \right)' =
  (2-n) \,  A_{ \beta} (r)  + r \,  A_{\beta} '(r) \, .
 \end{align}
  Using the second expression for $f(r)$, we compute its derivative (cf. \cite{CM3})
  \begin{align}
  	f'(r) = r^{1-n} \, \int_{u=r} \langle \nabla  \left( u^{2-n} \, |\nabla u|^{\beta} \right) , \nn \rangle \, .
 \end{align}
  Thus, for $r_1 < r_2$ the divergence theorem and 
  Corollary \ref{c:optimize} give
  \begin{align}	\label{e:gofr}
  	g (r_2) - g(r_1) &=
	r_2^{n-1} f' (r_2) -  r_1^{n-1} f'(r_1) = \int_{r_1 \leq u \leq r_2} \Delta \left( u^{2-n} \, |\nabla u|^{\beta} \right) 
		\notag \\
		&=   \beta \,   \int_{r_1 \leq u \leq r_2} 
	 u^{2-n} \, |\nabla u|^{\beta} \left( \left| \II_0 \right|^2 + \Ric (\nn , \nn)    \right) \\
	 &  + \frac{ \beta}{4(n-1)} \,  \int_{r_1 \leq u \leq r_2}
	 u^{-n} \, |\nabla u|^{\beta-2} 
	 \, \left(   \tilde{\beta} \, \left| B(\nn )\right|^2 + (n-2) \, \left| B(\nn)^T \right|^2
	\right) \notag \, ,
  \end{align}
 giving the first claim.
  
It follows from lemma $2.8$ of \cite{C} that there is a sequence $r_i \to 0$ so that
\begin{align}
	g(r_i) \to (2-n) \, \Vol (\partial B_1(0)) \, .
\end{align}
Putting this back into \eqr{e:gofr} and taking the limit as $r_i \to 0$ gives for $r > 0$ that
  \begin{align}	\label{e:gofr2q}
  	g(r) - (2-n)\, \Vol (\partial B_1(0))  =   \beta \,   \int_{u \leq r} 
	 u^{2-n} \, |\nabla u|^{\beta} \left( \left| \II_0 \right|^2 + \Ric (\nn , \nn)    \right) \notag \\
	 + \frac{ \beta}{4(n-1)} \,  \int_{  u \leq r}
	 u^{-n} \, |\nabla u|^{\beta-2} 
	 \, \left(   \tilde{\beta} \, \left| B(\nn )\right|^2 + (n-2) \, \left| B(\nn)^T \right|^2
	\right)\, .
  \end{align}
 The second claim follows from this since
 \begin{align}
 	 r^{n-1} \, \left(r^{2-n} \, \left( A_{ \beta} (r) -  \Vol (\partial B_1(0))\right)  \right)' 
	 = g(r) - (2-n)\, \Vol (\partial B_1(0))     \, .
 \end{align}
  \end{proof}
  
  \subsection{The third monotonicity formula}

   The following formula is the third monotonicity formula for $A_{\beta}$.  We will see that this leads to the monotonicity of $A_{\beta}$ itself.
  
  \begin{Thm}	\label{t:A2mono3}
If   $r_1 < r_2$, then
\begin{align}
r_2^{3-n} \, A_{\beta} '(r_2) &- r_1^{3-n} \, A_{\beta} '(r_1) 
= \beta \, \int_{ r_1 \leq u \leq r_2 } u^{4-2n} \,   |\nabla u|^{\beta} \, 
	 \left(           \left| \II_0 \right|^2  
	 +     \Ric (\nn , \nn) 
	  \right) \notag  \\
	  &+ \frac{\beta}{4(n-1)} \, \int_{ r_1 \leq u \leq r_2 } u^{2-2n}   \,  |\nabla u|^{\beta -2}   \left( \tilde{\beta} \,
	\left| B(\nn) \right|^2 + (n-2) \, \left| B(\nn)^T \right|^2 
	\right)  \, .
 \end{align}
\end{Thm}

\begin{proof} 
Using the formula \eqr{e:abetap} for $A_{\beta}'$, we  get
  \begin{align}	
 	r^{3-n} \, A_{\beta}'(r) &= r^{4-2n} \, \int_{ u = r } \langle \nabla   |\nabla u|^{\beta} , \nn \rangle  =
	 \beta \,  \int_{ u = r }  \langle u^{4-2n} \,  |\nabla u|^{\beta -1} \, \nabla   |\nabla u| , \nn \rangle \, ,
 \end{align}
 so that the divergence theorem gives 
  \begin{align}
 	r_2^{3-n} \, A_{\beta} '(r_2) - r_1^{3-n} \, A_{\beta} '(r_1)  = \beta \, \int_{ r_1 \leq u \leq r_2} \dv \, \left( u^{4-2n} \, |\nabla u|^{\beta -1} \nabla |\nabla u| \right)	   \, .
 \end{align}
 The theorem   follows from this since
the
second claim in Lemma \ref{l:divforms} with $\alpha = \beta - 1$ and $p=2-n$ gives
 \begin{align}	\label{e:thirdcl}
	&\dv \, \left( u^{4-2n} \, |\nabla u|^{\beta -1} \nabla |\nabla u| \right) =   
	u^{4-2n}   \,  |\nabla u|^{\beta}    \left(    \left| \II_0 \right|^2  + \Ric (\nn , \nn) \right) \notag \\
	&\qquad + \frac{1}{4(n-1)} \, u^{2-2n}   \,  |\nabla u|^{\beta -2}   \left( \tilde{\beta} \,
	\left| B(\nn) \right|^2 + (n-2) \, \left| B(\nn)^T \right|^2 
	\right) \, 
	  \, .
 \end{align}
\end{proof}

\subsection{Monotonicity of $A_{\beta}$}

Finally, we turn to the monotonicity of $A_{\beta}$.  
For this, we will focus on infinity, instead of the point singularity of the Green's function.

\begin{proof}[Proof of Theorem \ref{t:A2mono}]
We show first that $A_{ \beta} '(r) \leq 0$.
 It is convenient to define $f(r)$ by 
 \begin{align}	 
 	f(r) \equiv  r^{3-n} \, A_{\beta}'(r)  \, .
 \end{align}
 Theorem \ref{t:A2mono3} gives that
$f(s) \geq f(r)$ whenever $s> r$ and, thus, that
 \begin{equation}
 	A_{\beta}'(s) =  s^{n-3} \, f(s) \geq   s^{n-3} \, f(r)   \, .
 \end{equation}
 Integrating this from $r$ to $R$ would give
 \begin{align}
 	A_{\beta} (R) - A_{\beta} (r) \geq   f(r) \,  \int_r^R  s^{n-3}  \, ds \, .
 \end{align}
However, if $f(r) > 0$, then the right-hand side is not integrable as $R \to \infty$, but this contradicts 
 that $A_{{\beta}}(R)$ is uniformly bounded by Lemma \ref{l:bounded}.  This contradiction shows that we must always have $f(r) \leq 0$, completing the first step of the proof.
 
 The second step is to show that there is a sequence $r_j \to \infty$ with
 \begin{align}	\label{e:frj}
 	\left| f(r_j) \right| \to 0 \, .
 \end{align}
 However, this follows immediately from $A'$ having a sign and $|A|$ being bounded.
 
 Finally, the theorem follows from \eqr{e:frj}, Theorem  \ref{t:A2mono3} and the monotone convergence theorem.
 
\end{proof}
 
\begin{proof}[Proof of Corollary \ref{c:umb}]
By Theorem \ref{t:A2mono}, $A_1 (r)$ is non-increasing and
\begin{align}
	A_1(R/2) - A_1 (R) &\geq -\int_{ R/2}^R A_1'(r) \, dr \geq 
	\int_{ R/2}^R
	r^{n-3} \, \int_{ r \leq u  } u^{4-2n} \,    |\nabla u| 
	        \left| \II_0 \right|^2  \, dr \notag \\
	        &\geq   (2 R)^{2-n}  
	  \int_{ R \leq u \leq 2R }    |\nabla u| 
	        \left| \II_0 \right|^2 = 
	       (2R)^{2-n}  
	  \int_{ R}^{2R }   \int_{u=s}
	        \left| \II_0 \right|^2 \, ds \, ,
	     	\end{align}
	where the last equality is the coarea formula.

Since $A_{1}$ is nonnegative and non-increasing, it has a limit and, thus, 
\begin{align}
	\lim_{R \to \infty} \, \left\{ R^{2-n}  
	 \int_{ R}^{2R }   \int_{u=s}
	        \left| \II_0 \right|^2 \, ds \right\} = 0   \, . 
\end{align}
The corollary follows from this   since $\Vol (u=s) \geq s^{n-1} \, \Vol  (\partial B_1(0))$
by Lemma \ref{l:bounded}
\end{proof}

   \section{Examples}
     
     In this section we will give some simple examples that illustrate some of the results of the previous sections.
     
     \begin{Exa}
 Suppose that $M^2$ is a surface and $u:M\to \RR$ is a smooth function.  If $s\in \RR$ is a regular value of $u$, then the level set is umbilic since there is only one principal curvature.  
 On the other hand if $M$ is flat Euclidean space $\RR^n$ with $n\geq 3$ and $u:\RR^n\to \RR$ is a proper smooth function all of whose level sets are umbilic and $s$ is a regular value of $u$, then $u^{-1}(s)$ is a round sphere; see, for instance, \cite{S} Vol. IV, p. 11.    
 
 Obviously, there are many different functions on $\RR^n$ all of whose level sets are round spheres.   For instance, in addition to distance functions to fixed points, then the function
 \begin{align}
u(x)=\sqrt{|x|^2+x_1^2}-x_1\, 
\end{align}
has level sets that are round spheres
 \begin{align}
u^{-1}(s)=\{x\in \RR^n\, | \, |x-(s,0,\cdots,0)|^2=2\, s^2\}\,  .
\end{align}
This also shows that, even though all the level sets of $u$ are round spheres, the metric on $\RR^n$  may not be written as $dr^2+r^2\, d\theta^2$, where $\theta$ are coordinates on the level sets of $u$.
  \end{Exa}

  \begin{Exa}
 Let $M$ be the flat $4$-dimensional manifold $\RR^3\times \SS^1$ with coordinates $x=(x_1,x_2,x_3,\theta)$ and let $u_1$, $u_2:M\to \RR$ be nonnegative proper functions on $M$ given by
\begin{align}
u_1(x)&=\left(x_1^2+x_2^2+x_3^2\right)^{\frac{1}{4}}\, ,\\
u_2(x)&=\text{dist}_M(0,x)\, .
\end{align}
Then 
 \begin{align}
\liminf_{r\to \infty}\frac{r^{2}}{\Vol (u_1=\sqrt{r})}\int_{u_1=\sqrt{r}}|\II_0|^2&>0\, ,\\
\limsup_{r\to \infty}\frac{r^{2}}{\Vol (u_2=r)}\int_{u_2=r}|\II_0|^2&=0\, .
 \end{align}
  In particular, the level sets are asymptotically umbilic for $u_2$ but not for $u_1^2$.    Note that $u_2$ and $u_1^2$ are proportional to the distance function $r$ to $0$ when $r$ is large and thus, in particular,
 \begin{align}
\lim_{|x|\to \infty}\frac{u_1^2(x)}{u_2(x)}=1\, .
 \end{align}
 Note that both $u_1^{-2}$ and $u_2^{-2}$ (where smooth)
 are proper positive harmonic functions.  We saw   earlier   that, by Corollary \ref{c:umb}, for any proper positive harmonic function $u^{2-n}$ on a  manifold with nonnegative Ricci curvature and Euclidean volume growth
 \begin{align}
\liminf_{r\to \infty}\frac{r^{2}}{\Vol (u=r)}\int_{u=r}|\II_0|^2&=0\, .
 \end{align}
 \end{Exa}

\begin{Exa}
 Let $M$ be a smooth $n$-dimensional manifold with a warped product metric of the form
 \begin{align}
 dr^2+f^2(r,\theta)\,g\, ,
 \end{align}
 where $g$ is a metric on a smooth $(n-1)$-dimensional manifold $N$.  The second fundamental form of the level sets of $u=r$ is given by
 \begin{align}
 \II=\partial_r \log f\,g\, .
 \end{align}
 In particular, the level sets are umbilic.  However, if we also require that $\Delta r^2=2n\,|\nabla r|^2=2n$, then   $f=C\,r$ for some constant $C$ and thus $M$ with the metric is part of a metric cone.  To see this, note that
  \begin{align}
 \Delta r^2=2+2\,r\,\Delta r=2+2(n-1)\,r\,\partial_r \log f\, .
 \end{align}
 Thus if $\Delta r^2=2n\,|\nabla r|^2=2n$, then 
   \begin{align}
 1=r\,\partial_r \log f\, .
 \end{align}
 Or, in other words,
$
 f=C\,r\, .
$ \end{Exa}

\end{document}